\documentclass[a4wide,10pt,reqno]{article}
\usepackage{amsfonts}
\usepackage{amsmath}
\usepackage{amsthm}
\usepackage{amssymb}
\usepackage{lipsum}
\usepackage{mathrsfs}
\usepackage{hyperref}
\usepackage{enumerate}
\usepackage{array}
\usepackage{tabularx}
\usepackage{float}
\restylefloat{table}
\usepackage{amssymb,tikz}
\usepackage{caption}
\usepackage{subcaption}
\usepackage{amsmath}
\usepackage{amsthm}
\usepackage{amsbsy}
\usepackage{psfrag}
\usepackage{pstricks}
\usepackage{xcolor}
\usepackage{graphics}
\usepackage{graphicx}
\usepackage{epstopdf,bm}
\usepackage{fancyhdr}
\numberwithin{equation}{section}

\usepackage{graphicx}
\usepackage{enumitem}

\newcommand\norm[1]{\left\lVert#1\right\rVert}

\theoremstyle{plain}
\newtheorem{theorem}{Theorem}[section]
\newtheorem{defi}[theorem]{Definition}
\newtheorem{remark}[theorem]{Remark}

\newtheorem{lemma}[theorem]{Lemma}

\def\O{\Omega}

\def\cV{\mathcal{V}}

\def\cT{\mathcal{T}}
\def\cE{\mathcal{E}}

\def\mean#1{\left\{\hskip -5pt\left\{#1\right\}\hskip -5pt\right\}}
\def\jump#1{\left[\hskip -3.5pt\left[#1\right]\hskip -3.5pt\right]}
\def\smean#1{\{\hskip -3pt\{#1\}\hskip -3pt\}}
\def\sjump#1{[\hskip -1.5pt[#1]\hskip -1.5pt]}

\begin{document}
	\title{Crouzeix-Raviart Finite Element Approximation of Dirichlet Boundary Control Problems with Piecewise Constant Controls}
	\author{ Sudipto Chowdhury, Shallu \thanks{ Department of Mathematics, The LNM Institute of Information Technology, Jaipur, Rajasthan- 302031, India (\tt{sudipto.choudhary@lnmiit.ac.in, 22pmt@lnmiit.ac.in})}} 
    \maketitle
	\begin{abstract}
		\noindent
		This article examines the Dirichlet boundary control problem governed by the Poisson equation, where the control variables are square integrable functions defined on the boundary of a two-dimensional bounded, convex, polygonal domain. It employs an ultra-weak formulation and utilizes Crouzeix-Raviart finite elements to discretize the state variable, while employing piecewise constants for the control variable discretization. Furthermore, it establishes an optimal order \emph{\emph{a priori}} error estimate for the control variable. 
	\end{abstract}
	\everymath{\displaystyle}
    
\section{Introduction}

We consider the control-constrained Dirichlet boundary control problem on a convex and bounded polygonal domain $\O\subseteq \mathbb{R}^2$. Denote the polygonal boundary of the domain $\O$ by $\Gamma:=\cup_{j=1}^{k}\Gamma_j$, where $\Gamma_j$ is a straight line segment for all $1\leq j\leq k$. The minimization problem seeks a pair $(\bar{y},\bar{u})\in L_2(\O)\times U_{ad}$ such that
\begin{align}\label{Costfnl}
J(\bar{y},\bar{u})=\min_{(y,u)\in L_2(\O)\times U_{ad}} J(y,u),
\end{align}
subject to the following second-order elliptic equation:
\begin{align}\label{eqn:state}
\begin{cases}
-\Delta y&=0 \quad \text{in} \quad \O,\\
\hspace{.4cm}y&=u \quad \text{on} \quad \Gamma.
\end{cases}
\end{align}
$U_{ad} $ is a convex subset of $L_2(\Gamma)$ representing the set of admissible controls, which will be defined later. The quadratic cost functional 
 $J(\cdot,\cdot):L_2(\O)\times L_2(\Gamma)\rightarrow \mathbb{R}$ is defined in the following manner:
\begin{align}\label{cost:functional}
J(y,u):=\frac{1}{2}\int_{\O}|y-y_d|^2\,dx+\frac{\alpha}{2}\int_{\Gamma}|u|^2\,ds,
\end{align}
where $y_d\in L_2(\O)$ is the given desired state and $\alpha>0$ is known parameter which is introduced for the regularization. 
\noindent
The state equation \eqref{eqn:state} under the standard weak formulation is not well posed for $u\in L_2(\Gamma)$ as the weak solution of \eqref{eqn:state} lie in $H^1(\O)$ space and therefore by trace theorem \cite{brennerbook}, the first trace of it must lie in $H^{1/2}(\Gamma)$ which is not the case in this formulation. There are multiple techniques available in the literature to overcome this variational difficulty, one of which is to use the method of transposition which is also known as the ultra weak formulation \cite{casasraymond:2006,may:2013}. In brevity, we discuss the method of transposition. Define a test function space as follows:
\begin{align*}
X:=\{\phi \in H^1_0(\O):\Delta \phi \in L_2(\O)\}.
\end{align*}
Now, we obtain the ultra weak formulation of the state equation \eqref{eqn:state}. Multiplying $\phi\in X$ in both sides of \eqref{eqn:state} and then integrate both the sides to find 
\begin{align*}
-\int_{\O}\Delta y\,\phi\,dx&=0.
\end{align*}
A use of integration by parts in the above equation yields 
\begin{align*}
0&=\int_{\O}\nabla y\cdot\nabla \phi\,dx-\int_{\Gamma}\phi\,\frac{\partial y}{\partial n}\,ds\\
&=\int_{\O}\nabla y\cdot\nabla \phi\,dx,
\end{align*}
where we have used the fact that $\phi \in H^1_0(\O)$ to get the last equality. The idea in the method of transposition is to apply the integration by parts once more. By doing so, we get
\begin{align}\label{harmonic:extn}
0&=-\int_{\O}y\,\Delta \phi\,dx+\int_{\Gamma}y\,\frac{\partial \phi}{\partial n}\,ds.
\end{align}
Thus, the ultra weak formulation of \eqref{eqn:state} can be stated as: Given $u\in L_2(\Gamma)$, find $y\in L_2(\O)$ such that
\begin{align}\label{ultra:weak}
(y,\Delta \phi)&=\langle u,\partial \phi/\partial n\rangle \quad \forall ~ \phi \in X,
\end{align}
where $(\cdot,\cdot)$ and $\langle \cdot,\cdot\rangle$ denotes the $L_2$ inner product on the domain $\O$ and its boundary $\Gamma$ respectively. Another approach to study second order Dirichlet boundary control problem is to 
transform the Dirichlet boundary condition into Robin's type boundary condition \cite{casasmateosraymond:2009}. Other method is to employ the energy space based approach in which the control is penalized in the energy space $H^{1/2}(\Gamma)$ \cite{chowdhurykkt,phan:2015}.  In this article the control variable is discretized with piecewise constant finite elements while the state and adjoint states are discretized using first-order non conforming Crouzeix Raviart finite elements.\\ 
Finite element analysis for the Dirichlet boundary control problems addressed so far in the literature uses polynomials of same degrees to discretize the state and the control variables. In this article for the first time, to the best of our knowledge the control variables are discretized with piecewise constants whereas the first-order Crouzeix Raviart finite element functions are employed to discretize the state and adjoint state variables. Therefore the main challenge is to connect the discrete state and  the control variables. But here we make an important observation that CR elements are defined by degrees of freedom at the midpoints of edges. This aligns perfectly with piecewise constant controls on the boundary, as the constant control value interacts directly with the single midpoint value of the CR element described by \eqref{dse}. This is one of the main novelties of this article. This formulation also avoids complex integration of varying polynomial profiles. The optimal control
is found by projecting the normal derivative of the adjoint state onto a set of admissible controls. When the control
is discretized as piecewise constant, this projection becomes purely local and element-wise. Theorem \ref{estimate_for_boundary-extension} derives an estimate for the discrete boundary lifting of a piecewise constant function in $L_1$ norm  which reveals one important behaviour of the boundary lifting namely it's contribution concentrates near the boundary of the domain and not much in the interior of the domain. Also it's contribution becomes smaller proportional to the smallness of the discretization parameter. Theorem \ref{L2_norm_error_state_eq}, Remark \ref{rem1} discuss the $L_2$ norm error estimate for the CR discretization of the harmonic lifting of a piecewise constant function and establishes an almost linear rate of convergence. Theorem \ref{L2_norm_error_state_eq30} and Remark \ref{rem2} discuss the $L_2$ norm error estimate for the CR discretization of the harmonic lifting of a $H^{\frac{1}{2}}(\Gamma)$ function and establishes a linear order of convergence. These results are obviously super convergence results and posses  the theoretical importance on it's own in the theory of the CR finite element analysis. The points discussed above constitutes some of the key novelties of this article.\\
The outline of this article is as follows. In section 2 \eqref{Costfnl}-\eqref{eqn:state} is formulated in terms of the ultraweak formulation, the first-order optimality condition is derived and the regularity of the optimal control is derived. Section 3 contains the preliminaries of the first-order Cruzeix Raviart finite element discretization and the control space discretization. The discrete control problem is introduced in this section and the discrete optimality system is derived. This section contains an important representation of the discrete adjoint state variable in terms of the discrete control to state operator and it's adjoint. Section 4 contains an important $L_1$ norm estimate for the discrete boundary lifting of a piecewise constant function defined on the boundary of the domain and also contains some important auxiliary results regarding the $L_2$ norm error estimates of the harmonic lifting of a piecewise constant and $H^{\frac{1}{2}}(\Gamma)$ function. An almost optimal order a priori error estimate for the optimal control with respect to the $L_2$-norm is derived in Section 5, which is the main result of this article. Section 6 contains the numerical illustrations which justifies our theoretical findings.


\section{Continuous Problem}
In this segment, we discuss the variational setting of the investigated problem and derive the first order necessary optimality conditions. Before, we enter into the analysis part, we define the notations which are used throughout the article. The notations $H^m(\O)$, $H_0^m(\O)$, and $W^{m,p}(\O)$ represent the standard Sobolev spaces on $\O$, where $m$ is a non negative integer.  The Sobolev spaces of non integer order $s$, $H^s(\O)$ and $W^{s,p}(\O)$, are defined using interpolation, as explained in \cite{brennerbook}. On a Lipschitz domain, this definition is equivalent to the definition that uses double-integral norms, which is discussed in \cite[Theorem 14.2.3]{brennerbook}. The norms of $H^s(\Gamma)$ and $W^{s,p}(\Gamma)$, $0 \leq s \leq 1$ and $1 < p < \infty$, are defined on the boundary $\Gamma$ using charts, which is equivalent to using double-integral norms on $\Gamma$, according to \cite{brennerbook,Nezaetal}. The outwards unit normal vector to $\Gamma$ is denoted by $n$. The dual space of $H^{1/2}(\Gamma)$ is denoted by $H^{-1/2}(\Gamma)$ equipped with the operator norm $\|\cdot\|_{H^{-1/2}(\Gamma)}$ defined as:
\begin{align*}
\|\mu\|_{H^{-1/2}(\Gamma)}:=\sup\limits_{v\in H^{1/2}(\Gamma), v\neq 0}\frac{\langle \mu,v\rangle}{\|v\|_{H^{1/2}(\Gamma)}}, \quad \mu \in H^{-1/2}(\Gamma).
\end{align*}
Any generic member of $\mathbb{R}^2$ is denoted by $x=(x_1,x_2)$, where $x_1$, $x_2\in \mathbb{R}$. In the article, $C$ denotes a positive constant, which is generic in nature and does not depend on the solutions or the mesh-size.
\begin{remark}
When $v$ belongs to $H^2(\O)$, the trace of its gradient, denoted by $\nabla v|_{\Gamma}$, lies in $H^{1/2}(\Gamma)^2$. If $\Gamma$ has a smooth boundary, the outwards unit normal vector $n$ is continuous, allowing the definition of the normal derivative $\partial_n v = n\cdot\nabla v$ to be well defined. Thus $\partial_n v \in H^{1/2}(\Gamma)$ for $v \in H^2(\O)$ and the following estimate holds:
\begin{align*}
\|\partial _n v\|_{H^{1/2}(\Gamma)}\leq C\|v\|_{H^2(\O)}, \quad v\in H^2(\O).
\end{align*}
However, this estimate is not applicable for polygonal boundaries $\Gamma$, which are only Lipschitz continuous. Nonetheless, for $v \in H^2(\O)$, we still have $\partial_n v|_{\Gamma_j} \in H^{1/2}(\Gamma_j)$ for each straight component $\Gamma_j$, $1\leq j \leq k$ of $\Gamma$. To accommodate this, we introduce the space $\tilde{H}^{1/2}(\Gamma):=\{v \in L_2(\Gamma), v|_{\Gamma_j} \in H^{1/2}(\Gamma_j) \quad \forall~ 1\leq j \leq m\}$. We denote by $\tilde{H}^{-1/2}(\Gamma)$ the completion of $L_2(\Gamma)$ with respect to the following dual norm on $L_2(\Gamma)$
\begin{align*}
\|\mu\|_{\tilde{H}^{-1/2}(\Gamma)}:=\sup\limits_{v\in X\setminus\{0\}} \frac{\langle \mu,\partial_n v\rangle}{\|v\|_{H^2(\O)}}\leq \sup_{v\in \tilde{H}^{1/2}(\Gamma),v\neq 0}\frac{\langle \mu,v\rangle}{\|v\|_{\tilde{H}^{1/2}(\Gamma)}}.
\end{align*}
Note that $\tilde{H}^{-1/2}(\Gamma)$ is not in general the dual space of $\tilde{H}^{1/2}(\Gamma)$. In the case of a smooth boundary $\Gamma$, 
the mapping $\partial_n: X \rightarrow H^{1/2}(\Gamma)$ is surjective, so $\tilde{H}^{-1/2}(\Gamma)= H^{-1/2}(\Gamma)$.
\end{remark}
\noindent
For any $u\in \tilde{H}^{-1/2}(\Gamma)$, the harmonic lifting of $u$ is given  by $y\in L_2(\O)$, the solution of \eqref{harmonic:extn}.  
The function $u\in \tilde{H}^{-1/2}(\Gamma)$ is known as the ultra weak trace of $y_u\in L_2(\O)$ and the following trace estimate holds true:
\begin{align*}
\|u\|_{\tilde{H}^{-1/2}(\Gamma)}=\sup\limits_{\phi\in X\setminus\{0\}} \frac{(y_u,\Delta \phi)}{\|\phi\|_{H^2(\O)}}\lesssim \|y_u\|_{L_2(\O)}.
\end{align*}
Furthermore, if $u\in H^{1/2}(\Gamma)$, then $y_u\in H^1(\O)$ becomes the standard harmonic lifting of $u$ such that it satisfies the following equation:
\begin{align*}
(\nabla y_u,\nabla \phi)=0 \quad \forall ~ \phi \in V.
\end{align*}
In this case $u\in H^{1/2}(\Gamma)$ is simply the trace of $y_u$ in the sense that $y_u|_{\Gamma}=u$.
\subsection{Optimality System}
In this section, we write the model Dirichlet boundary optimal control problem and derive the corresponding optimality system by applying the first order necessary optimality conditions. For $u\in  L_2(\Gamma)$, define $y\in L_2(\O)$ to be the solutions of the following equations:
\begin{align}
(y,\Delta\phi)&=\langle u, \partial_n \phi\rangle \quad \forall ~ \phi \in X.\label{bdry:eqn}
\end{align}
 If the given data satisfy $u\in H^{1/2}(\Gamma)$, then we can recover the standard weak formulation, i.e. the following holds:
\begin{align}\label{eqn:weakstate}
a(y,v)=0\quad \forall ~v\in H^1_0(\O),
\end{align}
where $a(\cdot,\cdot):H^1(\O)\times H^1(\O)\rightarrow \mathbb{R}$ is a symmetric bilinear form given by:
\begin{align*}
a(v,w)&=\int_{\O}\nabla v\cdot \nabla w\,dx \quad \forall ~ v, w\in H^1(\O).
\end{align*}
The bilinear form $a(\cdot,\cdot)$ defines a continuous mapping on $H^1(\O)$ and coercive map on $H^1_0(\O)$. Set $V=H^1_0(\O)$ and $Q=H^1(\O)$. 
The subsequent lemma affirms the well-posed-ness of the boundary value problem \eqref{eqn:state} in its ultra weak form, and as a particular instance, it ensures the presence of the ultra weak harmonic lifting of the general boundary data $u \in \tilde{H}^{-1/2}(\Gamma)$.
\begin{lemma}\label{lem:existence}
For any given $u \in \tilde{H}^{-1/2}(\Gamma)$, the state equation in its ultra weak form \eqref{ultra:weak} has a unique solution $y\in L_2(\O)$ such that the following \emph{\emph{a priori}} estimate is satisfied:
\begin{align}\label{est:priori}
\|y\|_{L_2(\O)}\leq C\|u\|_{\tilde{H}^{-1/2}(\Gamma)},
\end{align}
where $H^{-2}(\O)$ is the dual space of $X$.
\begin{proof}
For $u \in H^{1/2}(\Gamma)$, there exists a unique weak solution $y\in H^1(\O)$ of \eqref{eqn:state} satisfying \eqref{eqn:weakstate}. An application of integration by parts in \eqref{eqn:weakstate} yields
\begin{align}\label{1.1}
-(y,\Delta \phi)&=-\langle u,\partial_n\phi\rangle \quad \forall ~ \phi \in X,
\end{align}
which implies that $y$ is the solution of \eqref{ultra:weak}. We apply duality argument to prove the \emph{\emph{a priori}} error estimate \eqref{est:priori}. Consider $z\in V$ to be the weak solution of the following Poisson problem:
\begin{align*}
\begin{cases}
-\Delta z&=y \quad \text{in} \quad \O,\\
\hspace{.4cm}z&=0 \quad \text{on} \quad \Gamma.
\end{cases}
\end{align*}
By elliptic regularity theory, $z\in H^2(\O)$ and $\|z\|_{H^2(\O)}\leq C \|y\|_{L_2(\O)}$. Thus, using \eqref{1.1} for $\phi=z$, we get
\begin{align*}
\|y\|^2_{L_2(\O)}&=(y,-\Delta z)\\
&=-\langle u,\partial _n z\rangle,
\end{align*}
which on using Cauchy-Schwarz inequality yields \eqref{est:priori}. Since and $H^{1/2}(\Gamma)\subseteq \tilde{H}^{-1/2}(\Gamma)$ are dense, by density arguments, for $u\in \tilde{H}^{-1/2}(\Gamma)$  the equation \eqref{ultra:weak} possesses a unique solution $y\in L_2(\O)$ satisfying the \emph{\emph{a priori}} error estimate \eqref{est:priori}.
\end{proof}
\end{lemma}
\noindent
In light of the Lemma \ref{lem:existence}, we define the control to state map which maps the given data to the unique solution of \eqref{ultra:weak}.
\begin{defi}\label{control:to:state}
For $u\in  L_2(\Gamma)$, the solution map $S:L_2(\Gamma)\rightarrow L_2(\O)$ is defined as $Su=y$ such that $y$ satisfy \eqref{bdry:eqn}.
\end{defi}
\noindent
Now, we define the set of admissible controls $U_{ad}$ as follows: 
\begin{align*}
U_{ad}:=\{u\in L_2(\Gamma):u_a\leq u(x)\leq u_b ~~ \text{a.e} ~~ x\in \Gamma\},
\end{align*}
where $u_a$, $u_b$ $\in \mathbb{R}$ are such that $u_a<0<u_b$. The model Dirichlet boundary optimal control problem reads as: find $(\bar{y},\bar{u})\in L_2(\O)\times U_{ad}$ such that
\begin{align}\label{mocp}
J(\bar{y},\bar{u})=\min_{(y,u)\in L_2(\O)\times U_{ad}}J(y,u),
\end{align}
subject to the constraint $y=Su$, where $J(\cdot,\cdot)$ is the quadratic cost functional defined by \eqref{cost:functional}. Next, we present the following result that addresses the well-posedness of the underlying optimal control problem \eqref{mocp} and the corresponding first-order optimality system.
\begin{theorem}
The optimal control problem \eqref{mocp} possess a unique solution $({y},\bar{u})\in L_2(\O)\times U_{ad}$. Moreover the Gelfund triplet $({y}, p,\bar{u})$ solves the following first order necessary optimality conditions:
\begin{align}
&(y_{\bar{u}},\Delta \phi)=\langle \bar{u}, \partial_n \phi\rangle \quad \forall ~ \phi \in X,\label{ctskkt:3}\\
&(\nabla v,\nabla p)=({y}-y_{d},v)\quad \forall~ v\in V,\label{As1}\\
&(\alpha\bar{u}-\dfrac{\partial p}{\partial n},u-\bar{u})\geq 0\quad \forall ~ u \in U_{ad}.\label{ctskkt:4}
\end{align}
\begin{proof}
By utilizing the solution operator $S(\cdot,\cdot)$, we can reformulate the model optimal control problem \eqref{mocp} into a simplified form, expressed as follows:
\begin{align}\label{red:mocp}
\min_{u\in U_{ad}}j(u),
\end{align}
where $j(\cdot):L_2(\Gamma)\rightarrow \mathbb{R}$ represents the reduced cost functional defined as:
\begin{align*}
j(u):=\frac{1}{2}\|Su-y_d\|^2_{L_2(\O)}+\frac{\alpha}{2}\|u\|^2_{L_2(\Gamma)}.
\end{align*}
The reduced cost functional $j(\cdot)$ defines a strictly convex functional on the Hilbert space $L_2(\Gamma)$. The set of admissible controls $U_{ad}$ is a nonempty, closed, convex subset of $L_2(\Gamma)$, thus the standard theory of PDE constrained optimal control problem \cite{trolzbook} implies that problem \eqref{red:mocp} has a unique solution say $\bar{u}\in U_{ad}$. Setting $\bar{y} = S(f, \bar{u})$ as the corresponding state variable. By definition of $S(\cdot,\cdot)$, we find that $\bar{y}\in L_2(\O)$ satisfies \eqref{ctskkt:3}. By applying the first order necessary optimality conditions, we get for any $u\in U_{ad}$
\begin{align}\label{cost_fnl_vi}
0\leq  j'(\bar{u})(u-\bar{u})&:=\lim_{t\downarrow 0}\frac{j(\bar{u}+t(u-\bar{u}))-j(\bar{u})}{t}\\
&=({y}-y_d,Su-S{\bar{u}})+\alpha\langle\bar{u},u-\bar{u}\rangle,
\end{align}
Introduce an adjoint state $p\in V$ such that it satisfies \eqref{As1} 
Using integration by parts and density arguments in \eqref{As1}, we get
\begin{align}\label{As2}
(v,\Delta p)&=-({y}-y_d,v)\quad \forall~ v\in L_2(\O).
\end{align}
Since $p\in H^2(\O)$, it implies $\partial _n p \in \tilde{H}^{1/2}(\Gamma)$. 
Choosing $v=Su-S\bar{u}$ in \eqref{As2} and applying integration by parts, the trace of $Su-S\bar{u}$ on $\Gamma$ implies
\begin{align*}
(\bar{y}-y_d,Su-S{\bar{u}})&=-(Su-S{\bar{u}},\Delta p)\\
&=-\langle u-\bar{u},\dfrac{\partial p}{\partial n}\rangle,
\end{align*}
then \eqref{cost_fnl_vi} implies $\bar{u}\in U_{ad}$ solves \eqref{ctskkt:4}. This completes the proof.
\end{proof}
\end{theorem}
\subsection{Regularity of the Optimal Control} In this fragment, we obtain the regularity of the optimal variables namely $\bar{y}$ and $\bar{u}$ which are useful for the subsequent analysis. 
Then $p\in X$ and by elliptic regularity theory for convex polygonal domains \cite{grisvardbook},   it satisfies the following \emph{\emph{a priori}} estimate:
\begin{align*}
\|p\|_{H^2(\O)}&\leq C \|\bar{y}-y_d\|_{L_2(\O)}.
\end{align*}
For any $u\in U_{ad}$, using \eqref{As2} for $v=y_u-y_{\bar{u}}\in L_2(\O)$, we get 
where we have used \eqref{harmonic:extn} to observe the last equality. Hence, the inequality \eqref{ctskkt:4} reduces to 
\begin{align*}
\alpha \langle \bar{u},u-\bar{u}\rangle-\langle u-\bar{u},\dfrac{\partial p}{\partial n}\rangle&\geq 0 \quad \forall ~ u\in U_{ad},
\end{align*}
which implies
\begin{align}\label{ctskkt:5}
\langle \alpha \bar{u}-\partial _n p,u-\bar{u}\rangle \geq 0 \quad \forall ~ u\in U_{ad}.
\end{align}
It follows from \eqref{ctskkt:5} and standard arguments from \cite[Chapter 2]{trolzbook} that
\begin{align}\label{ctskkt:6}
\bar{u}(x)=P_{[u_a, u_b]}\Big(\dfrac{1}{\alpha}\dfrac{\partial p}{\partial n} (x)\Big),\quad \text{for almost every} ~ x\in \Gamma,
\end{align}
where $P_{[u_a, u_b]}(w):=\min\{u_b,\max\{u_a,w\}\}$ denotes the projection of $\mathbb{R}$ onto $[u_a,u_b]$. Since $\Omega$ is convex, therefore it follows from \eqref{ctskkt:6} that $\bar{u}\in \tilde{H}^{1/2}(\Gamma)$. As $p$ equals zero on $\Gamma$, the derivative of $\theta$ in the tangential direction also equals zero on $\Gamma$. Therefore, at the corner points of the domain denoted as $x$, we have $\dfrac{\partial p}{\partial n}(x)=0$. Hence, it follows from \eqref{ctskkt:6} that $\bar{u}(x)=0$ at all the corner points $x$ of the domain, which implies that $\bar{u}\in H^{1/2} (\Gamma)$ from the double integral definition of the fractional derivative of $H^{1/2} (\Gamma)$.

\section{Discretization }
This section is devoted to the discretization.
\subsection{Preliminaries}
We introduce the following notations which are used throughout the article.
\begin{itemize}
\item $\cT_h$ is a regular triangulation of $\O$ into closed triangles $T$.
\item $h_T:=$ diam ($T$) is the diameter of the triangle $T\in \cT_h$.
\item $h:=\max_{T\in \cT_h}h_T$ is the mesh parameter.
\item $\cV_h^i:=$ the set of all interior vertices of $\cT_h$.
\item $\cV_h^b:=$ the set of all boundary vertices of $\cT_h$.
\item $\cV_h:=\cV_h^i\cup \cV_h^b$ is the set of all vertices of $\cT_h$.
\item $\cE_h^i:=$ the set of all interior edges of $\cT_h$.
\item $\cE_h^b:=$ the set of all boundary edges of $\cT_h$.
\item $\cE_h:=\cE_h^i\cup \cE_h^b$ is the set of all edges of $\cT_h$.
\item $m_e:=$ mid point of $e\in \cE_h$.
\item $\mathbb{P}_k(T):=$ the set of all polynomials of degree atmost $k\in \mathbb{N}\cup \{0\}$ over $T\in \cT_h$.
\item $H^k(\O,\cT_h):=\{v\in L_2(\O):v|_{T}\in H^k(T) \quad \forall~ T\in \cT_h\}$, $k\in \mathbb{N}\cup \{0\}$.
\end{itemize} 
From this point forward, the notation $a\lesssim b$ indicates the existence of a positive constant $C$ that is not dependent on the mesh parameter $h$, such that a $a\leq Cb$. It is assumed that the triangulation $\cT_h$ is regular, meaning that any two distinct simplices in $\cT_h$ with non-empty intersection are either identical, or share exactly one common vertex, or one common edge. Additionally, every interior angle of any simplex is bounded from below by a universal positive constant $\rho$. All the generic constants hidden in the notation $\lesssim$ are solely determined by $\rho>0$, thereby ensuring the triangulation is shape regular.\\
Next, we will establish the definitions for the jump and average of functions that are scalar-valued and vector-valued. Consider an edge $e \in \cE_h^i$ that is shared by two adjacent triangles $T_{+}$ and $T_{-}$, such that $e = T_{+}\cap T_{-}$. Let $n_{+}$ be the unit normal of $e$ pointing from $T_{+}$ to $T_{-}$, and $n_{-}= -n_{+}$. In this context, we define the jumps $\sjump{\cdot}$ and averages $\smean{\cdot}$ across the edge $e$ as follows:
\begin{itemize}
\item For a scalar valued function $w\in H^1(\O,\cT_h)$, define
\begin{align*}
\sjump{w}:=w_{+}\,n_{+}+w_{-}\,n_{-}, \quad \quad \smean{w}:=\frac{w_{+}+w_{-}}{2},
\end{align*}
where $w_{\pm}=w|_{T_{\pm}}$.
\item For a vector valued function $v\in [H^1(\O,\cT_h)]^2$, define
\begin{align*}
\sjump{v}:=v_{+}\cdot n_{+}+v_{-}\cdot n_{-}, \quad \quad \smean{v}:=\frac{v_{+}+v_{-}}{2},
\end{align*}
where $v_{\pm}=v|_{T_{\pm}}$.
\end{itemize}
To simplify the notation, we also introduce the concepts of jump and mean on the boundary $\Gamma$. Consider any edge $e \in \cE_h^b$, where it is evident that there exists a triangle $T \in \cT_h$ such that $e = \partial T \cap \Gamma$. Let $n_e$ be the unit normal of $e$ pointing outward from $T$. For any $w \in H^1(\O,\cT_h)$ and any $v \in [H^1(\O,\cT_h)]^2$, we define the following on $e \in \cE_h^b$:
\begin{align*}
\sjump{w}&:=w\,n_e \quad \text{and} \quad \smean{w}=w,\\
\sjump{v}&:=v\cdot n_e \quad \text{and} \quad \smean{v}=v.
\end{align*}
In order to approximate the state variables we define the Crouziex-Raviart finite element spaces. 
\begin{align*}
V_h&:=\{v_h\in L_2(\O):v_h|_{T}\in \mathbb{P}_1(T) ~ \forall ~ T\in \cT_h ~ \text{and} ~ v_h ~\text{is continuous at}~ m_e ~\forall ~e\in \cE_h \},\\
V_h^0&:=\{v_h\in V_h: v_h(m_e)=0 ~\forall ~ e\in \cE_h^b \}.
\end{align*}
The discrete bilinear form $a_{pw}(\cdot,\cdot):V_h\times V_h\rightarrow \mathbb{R}$ is defined by
\begin{align*}
a_{pw}(v_h,w_h)&=\sum_{T\in \cT_h}\int_{T}\nabla v_h|_T\cdot\nabla w_h|_T\,dx \quad \forall~ v_h,w_h\in V_h. 
\end{align*}
The mesh dependent bilinear form $a_{pw}(\cdot,\cdot)$ defines a symmetric bilinear form which coincides with the continuous bilinear form  $a(\cdot,\cdot)$ on $ H^1(\O)\times H^1(\O)$. Now, we define the mesh dependent norm on $V_h$. For any $v_h\in V_h$, define $\|\cdot\|_h$ as follows:
\begin{align*}
\|v_h\|_{h}:=\sqrt{a_{pw}(v_h,v_h)}.
\end{align*}
Note that, $\|\cdot\|_h$ defines only a semi-norm on $V_h$ but a norm on $V_h^0$. Next, we define a Crouziex-Raviart interpolation map $I_{CR}:V\cap C(\bar{\O})\rightarrow V_h^0$ as follows:
\begin{align}\label{CR;inter}
\int_{e}I_{CR}v\,ds=\int_{e}v\,ds\quad \forall ~ e\in \cE_h, ~ v\in V\cap C(\bar{\O}).
\end{align}
By density arguments, $I_{CR}$ can be extended continuously and uniquely to $V$, and the extension is again denoted by $I_{CR}$ for the ease of presentation. In the next lemma, we state the approximation properties of $I_{CR}$ \cite{CRetal}.
\begin{lemma}\label{lem:CRinterest}
For $v\in V$, it holds that:
\begin{align*}
\|v-I_{CR}v\|_{L_2(\O)}+h\|v-I_{CR}v\|_{h}&\lesssim h|v|_{H^1(\O)}.
\end{align*}
\end{lemma}
\noindent
\noindent
\subsection{Discrete Control Problem}
This section contains the discrete control problem and the main result of this article. The control variable is discretized with the functions from $U_h.$ Define the discrete control problem by:
\begin{align}\label{dcp}
J_{h}(w_{h},u_{h})=\dfrac{1}{2}\|w_{h}-y_{d}\|^{2}+\dfrac{\alpha}{2}\|u_h\|_{L_2(\Gamma)}^{2},
\end{align}
subject to,
\begin{align}\label{dse}
&a_{pw}(w_h,v_h)=0,~\forall v_h\in V^0_h, ~\text{and}~\int_ew_hds=\int_eu_hds,~\text{where}~e\in\mathcal{E}^b_h.
\end{align}
Note that in order to define the discretization of \eqref{eqn:state} it's enough to specify the Cruzeix Raviart degrees of freedoms of $y_h$ on the boundary edges.

\noindent
\vspace{1mm}
Define the discrete control to state operator
\begin{align}\label{discrete_control_to_state}
    S_h: U_h\rightarrow V_h,~~~~S_h(u_h)=w_h.
\end{align}
Then the reduced problem is 
\begin{equation}\label{rdcp}
    \min_{u_h\in U_{h,ad}}j_h(u_h),
\end{equation}
where 
\begin{align}
    j_{h}(u_{h})=\dfrac{1}{2}\|S_hu_{h}-y_{d}\|^{2}+\dfrac{\alpha}{2}\|u_h\|_{L_2(\Gamma)}^{2}
\end{align}
The following theorem discusses the discrete optimality system and the existence and uniqueness of the discrete optimal variables for \eqref{dcp}-\eqref{dse}.

\begin{theorem}
There exists unique $({y}_h, p_h, \bar{u}_h)\in V_h\times V_h^0\times U_{h,ad}$ of the discrete optimal control problem. Furthermore, it satisfy the following discrete optimality system:
\begin{align}
&a_{pw}(y_h,v_h)=0,~\forall v_h\in V^0_h, ~\text{and}~\int_ey_hds=\int_eu_hds,~\text{where}~e\in\mathcal{E}^b_h.\label{diskkt:1}\\
&a_{pw}(p_h, v_h)=(y_h-y_d,v_h) \quad \forall ~ v_h \in V_h^0,\label{diskkt:2}\\
&\langle \alpha \bar{u}_h-\partial _n p_h,u_h-\bar{u}_h\rangle \geq 0 \quad \forall ~ u_h\in U_{h,ad}.\label{diskkt:3}
\end{align}
\end{theorem}
In the next lemma we discuss a crucial property for $S_h$ the discrete control to state operator.
\begin{lemma}\label{discrete_solution_operator_representation}
    For the discrete control to state operator $S_h:U_h\rightarrow V_h$ the following relation holds:
    \begin{equation*}
        \int_{\Gamma}\dfrac{\partial q_h}{\partial n}u_hds=-(S^*_h(S_hu_h-y_d),u_h)+(w_h-y_d,\tilde{w}_h)~~\forall u_h\in U_h,
    \end{equation*}
where $\tilde{w}_h\in V_h$ is the boundary lifting of $u_h\in U_h$  and $q_h\in V^0_h$ satisfies:

\vspace{1 mm}
Find $q_h\in V^0_h$ such that
\begin{align*}
    a_{pw}(q_h, v_h)=(w_h-y_d,v_h) \quad \forall ~ v_h \in V_h^0,
\end{align*}
where $w_h=S_hu_h.$
\end{lemma}
\begin{proof}
    Let $\tilde{w}_h\in V_h$ be the boundary lifting of $u_h$ $i.~e.$ all the degrees interior degrees of freedoms of $\tilde{w}_h$ are equal to zero and . Let $w_h=w^0_h+\tilde{w}_h$ where $w^0_h\in V^0_h$ satisfies,

    \vspace{1mm}
    Find $w^0_h\in V^0_h$ such that
\begin{align}\label{auxiliary_boundary-lifting}
    a_{pw}(w^0_h,v_h)=-a_{pw}(\tilde{w}_h,v_h)~~\forall v_h\in V^0_h.
\end{align}
Consider, 
\begin{align}\label{derivation_discrete_soln_opr1}
    a_{pw}(w_h,q_h)=&a_{pw}(w^0_h+\tilde{w}_h,q_h)\notag\\
    =&(w_h-y_d,w^0_h)+a_{pw}(q_h,\tilde{w}_h)\notag\\
    =&(w_h-y_d,y^0_h)+\sum_T(\nabla q_h,\nabla \tilde{w}_h)\notag\\
    =&(w_h-y_d,y^0_h)+ \left[\sum_T-(\Delta q_h,\tilde{w}_h)_T+(\dfrac{\partial q_h}{\partial n},\tilde{w}_h)_{\partial T}\right]\notag\\
    =&(w_h-y_d,w^0_h)+\sum_{e\in\mathcal{E}^i_h}\int_e\jump{\dfrac{\partial q_h}{\partial n}}\mean{\tilde{w}_h}ds+\notag\\&\sum_{e\in\mathcal{E}^i_h}\int_e\mean{\dfrac{\partial q_h}{\partial n}}\jump{\tilde{w}_h}ds+\sum_{e\in\mathcal{E}^b_h}\int_e\dfrac{\partial q_h}{\partial n}\bar{u}_hds.
    \end{align}
    Note that $\mean{\dfrac{\partial q_h}{\partial n}}$ is constant on each edge $e$, also note that on an interior edge $e$, $\int_e\jump{\tilde{y}_h}ds=0$. Therefore
    \begin{align}\label{derivation_discrete_soln_opr2}
        \int_e\mean{\dfrac{\partial q_h}{\partial n}}\jump{\tilde{w}_h}ds=\mean{\dfrac{\partial q_h}{\partial n}}\int_e\jump{\tilde{w}_h}ds=0
    \end{align}
    for $e\in\mathcal{E}^i_h$.
    Also note that as $\tilde{w}_h$ is the boundary lifting of $u_h$ therefore the degrees of freedoms on all the interior edges which is the value of the function at the midpoints are zero. Therefore on  all the interior edges $e$ 
    \begin{align}\label{derivation_discrete_soln_opr3}
      \int_e\mean{\tilde{w}_h}ds=0.  
    \end{align}
    Therefore
\begin{align}\label{derivation_discrete_soln_opr4}
    \int_e\jump{\dfrac{\partial q_h}{\partial n}}\mean{\tilde{w}_h}ds=\jump{\dfrac{\partial q_h}{\partial n}}\int_e\mean{\tilde{w}_h}ds=0.
\end{align}
Also note that $q_h\in V^0_h$, therefore
\begin{align}\label{derivation_discrete_soln_opr5}
    a_{pw}(q_h,w_h)=0.
\end{align}
Therefore combining \eqref{derivation_discrete_soln_opr1}, \eqref{derivation_discrete_soln_opr2}, \eqref{derivation_discrete_soln_opr3}, \eqref{derivation_discrete_soln_opr4} and \eqref{derivation_discrete_soln_opr5} we obtain
\begin{align}
    \int_{\Gamma}\dfrac{\partial q_h}{\partial n}{u}_hds&=-(S_hu_h-y_d,S_hu_h)+(w_h-y_d,\tilde{y}_h)\notag\\
     \int_{\Gamma}\dfrac{\partial q_h}{\partial n}{u}_hds&=-(S^*_h(S_hu_h-y_d),u_h)+(w_h-y_d,\tilde{w}_h)
\end{align}
\end{proof}

\section{Auxiliary Results}
This section contains some auxiliary results which play a very significant role for the rest of the analyses.
We begin by deriving a stability estimate for the boundary extension $\tilde{w}_h\in V_h$ introduced in Lemma \ref{discrete_solution_operator_representation}.
\begin{theorem}\label{estimate_for_boundary-extension}
    For the boundary lifting $\tilde{w}_h\in V_h$ the following stability estimate holds:
    \begin{equation*}
        \|\tilde{w}_h\|_{L_1(\Omega)}\leq Ch\|u_h\|_{L_1(\Gamma)}.
    \end{equation*}
\end{theorem}
\begin{proof}
    Let $\{\phi_e\}_{e\in\mathcal{E}_h}$ be the standard CR basis for $V_h$.
    Since for the boundary extension $\tilde{w}_h$ all the interior degrees of freedoms are zero therefore
    \begin{align}\label{estimate_for_boundary-extension_1}
        \tilde{w}_h=&\sum_{e\in\mathcal{E}^b_h}\tilde{w}_h(m_e)\phi_e\notag\\
        =&\sum_{e\in\mathcal{E}^b_h}u_h|_e\phi_e
    \end{align}
    Note that \eqref{estimate_for_boundary-extension_1} implies
    \begin{align}\label{estimate_for_boundary-extension_2}
        \|\tilde{w}_h\|_{L_1(\Omega)}=&\int_{\Omega}|\sum_{e\in\mathcal{E}^b_h}\tilde{w}_h(m_e)\phi_e(x)|dx\notag\\
        =&\int_{\Omega}|\sum_{e\in\mathcal{E}^b_h}u_h|_e\phi_e(x)|dx.
    \end{align}
Note that for $e_1,e_2\in\mathcal{E}^b_h$, $e_1\neq e_2$ implies $\text{supp}(\phi_{e_1})$ and $\text{supp}(\phi_{e_2})$ can have at most a line in common which is of measure zero with respect to the standard Lebesgue measure in $\mathbb{R}^2$. Therefore 
\eqref{estimate_for_boundary-extension_2} implies
\begin{align}\label{estimate_for_boundary-extension_3}
    \|\tilde{w}_h\|_{L_1(\Omega)}=&\int_{\Omega}|\sum_{e\in\mathcal{E}^b_h}u_h|_e\phi_e|_{\text{int}T_e}(x)|dx\notag\\
    =&\int_{\Omega}\sum_{e\in\mathcal{E}^b_h}|u_h|_e\phi_e|_{\text{int}T_e}(x)|dx\notag\\
    =&\sum_{e\in\mathcal{E}^b_h}\int_{\Omega}|u_h|_e\phi_e|_{\text{int}T_e}(x)|dx\notag\\
    =&\sum_{e\in\mathcal{E}^b_h}\int_{T_e}|u_h|_e\phi_e|_{\text{int}T_e}(x)|dx\notag\\
    =&\sum_{e\in\mathcal{E}^b_h}|u_h|_e|\int_{T_e}|\phi_e|_{\text{int}T_e}(x)|dx\notag\\
    \leq&\sum_{e\in\mathcal{E}^b_h}|u_h|_e|h^2_e\notag\\
    \leq& Ch\sum_{e\in\mathcal{E}^b_h}|u_h|_e|h_e.
\end{align}
So it remains to show that $\sum_{e\in\mathcal{E}^b_h}|u_h|_e|h_e=\|u_h\|_{L_1(\Gamma)}$.
\begin{align}\label{estimate_for_boundary-extension_4}
    \|u_h\|_{L_1(\Gamma)}=\int_{\Gamma}|\sum_{e\in\mathcal{E}^b_h}u_h|_e\lambda_e(s)|ds
\end{align}
Note that for $e_1,e_2\in\mathcal{E}^b_h,~\text{supp}T_{e_1}\cap\text{supp}T_{e_2}$ is at most a point that is the common point of $e_1$ and $e_2$, which has zero measure with respect to the standard Lebesgue measure on $\mathbb{R}$.
Therefore from \eqref{estimate_for_boundary-extension_4} we obtain
\begin{align}\label{estimate_for_boundary-extension_5}
    \|u_h\|_{L_1(\Gamma)}=&\int_{\Gamma}|\sum_{e\in\mathcal{E}^b_h}u_h|_e\lambda_e|_{\text{int}(e)}(s)|ds\notag\\
    =&\int_{\Gamma}\sum_{e\in\mathcal{E}^b_h}|u_h|_e\lambda_e|_{\text{int}(e)}(s)|ds\notag\\
    =&\sum_{e\in\mathcal{E}^b_h}\int_{\Gamma}|u_h|_e\lambda_e|_{\text{int}(e)}(s)|ds\notag\\
    =&\sum_{e\in\mathcal{E}^b_h}|u_h|_e|_{\text{int}(e)}\lambda_e|{\int_{e}}(s)|ds\notag\\
    =&\sum_{e\in\mathcal{E}^b_h}|u_h|_e|h_e\notag\\
\end{align}
Therefore combining \eqref{estimate_for_boundary-extension_3} and \eqref{estimate_for_boundary-extension_5} we obtain the desired estimate.
\end{proof}
In the next theorem we discuss the $L_2$ norm error estimate for the CR approximation of the solution of the state equation. We see that we are able to obtain a linear convergence rate for this which is obviously a super convergence result.
\begin{theorem}\label{L2_norm_error_state_eq}
    We have $|(y-y_d,(S-S_h)(P_0u-u_h))|\leq Ch^{1-\delta}$ for any $\delta>0$.
\end{theorem}
\begin{proof}
    Consider \eqref{As1} in the strong form $i.~e.$
    \begin{align}\label{As1_strong_form}
        -\Delta p&=y-y_d~~~~\text{in}~~\Omega,\\
        p&=0~~~~\text{on}~~\Gamma.\notag
    \end{align}
Consider 
\begin{align}\label{L2_norm_error_state_eq1}
    -\int_{\Omega}\Delta p(S-S_h)(P_0u-u_h)dx=&\sum_{T}\int_T\nabla p\cdot \nabla((S-S_h)(P_0u-u_h))dx-\notag\\
    &\sum_{T}\int_{\partial T}\dfrac{\partial p}{\partial n}\jump{(S-S_h)(P_0u-u_h)}ds\notag\\
    =&\sum_{T}\int_T\nabla (p-I_{CR}p)\cdot \nabla((S-S_h)(P_0u-u_h))dx+\notag\\
    &\sum_{T}\int_T\nabla I_{CR}p\cdot \nabla((S-S_h)(P_0u-u_h))dx-\notag\\
    &\sum_{e}\int_{e}\mean{\dfrac{\partial p}{\partial n}}\jump{(S-S_h)(P_0u-u_h)}ds\notag\\
    =&\sum_{T}\int_T\nabla (p-I_{CR}p)\cdot \nabla((S-S_h)(P_0u-u_h))dx+\notag\\
    &\sum_{T}\int_T\nabla I_{CR}p\cdot \nabla((S-S_h)(P_0u-u_h))dx-\notag\\
    &\sum_{e}\int_{ e}\mean{\dfrac{\partial p}{\partial n}-\dfrac{\partial I_{CR}p}{\partial n}}\jump{(S-S_h)(P_0u-u_h)}ds
\end{align}

Therefore it remains to estimate the terms on the right hand side of \eqref{L2_norm_error_state_eq1}. Consider the second term 
\begin{align}\label{L2_norm_error_state_eq2}
    \sum_{T}\int_T\nabla I_{CR}p\cdot \nabla((S-S_h)(P_0u-u_h))dx=&-\sum_T\int_T\Delta S(P_0u-u_h)I_{CR}pdx+\notag\\
    &\sum_T\int_{\partial T}\dfrac{\partial S(P_0u-u_h)}{\partial n}I_{CR}pds\notag\\
    =&\sum_{_e\in\mathcal{E}_h}\dfrac{\partial S(P_0u-u_h)}{\partial n}\jump{I_{CR}p-p}ds
    \end{align}
    Note that since $P_0u-u_h$ is piecewise constant therefore $S(P_0u-u_h)$ has singularity near the the boundary points where  $P_0u-u_h$ has jumps.  Therefore the regularity theory of the harmonic lifting implies that near a singular point, $|\nabla S(P_0u-u_h)|$ grows in the order of $\dfrac{1}{r}$ if we choose polar coordinate representation in a neighborhood of the singular point with the singular point as the origin.. Therefore $S(P_0u-u_h)\notin H^1(\Omega)$ instead $S(P_0u-u_h)\in W^{1,2-\epsilon}(\Omega)$ in two dimensions, for any $\epsilon>0$. Therefore,
   \begin{align}\label{L2_norm_error_state_eq3}
       \sum_{_e\in\mathcal{E}_h}\int_e\dfrac{\partial S(P_0u-u_h)}{\partial n}\jump{p-I_{CR}p}ds\leq&[\sum_{_e\in\mathcal{E}_h}\|\dfrac{\partial S(P_0u-u_h)}{\partial n}\|_{W^{-\frac{1}{2-\epsilon},2-\epsilon}(e)}][\sum_{_e\in\mathcal{E}_h}\|\jump{p-I_{CR}p}\|_{W^{\frac{1}{2-\epsilon},\frac{2-\epsilon}{1-\epsilon}}(e)}]\notag\\
       \leq&C[\sum_{_e\in\mathcal{E}_h}h_e^{\frac{-1}{2-\epsilon}}\|\nabla S(P_0u-u_h)\|_{L_{2-\epsilon}(T_e)}][\sum_{_e\in\mathcal{E}_h}[h_e^{\frac{-1}{2-\epsilon}}\|p-I_{CR}p\|_{L_{\frac{2-\epsilon}{1-\epsilon}}(T_e)}+\notag\\
       &h_e^{\frac{1}{2-\epsilon}}\|\nabla(p-I_{CR}p)\|_{L_{\frac{2-\epsilon}{1-\epsilon}}(T_e)}]]
   \end{align}
    Then application of the interpolation error estimate and Holder's inequality in \eqref{L2_norm_error_state_eq3} implies 
    \begin{align}\label{L2_norm_error_state_eq4}
    \sum_{_e\in\mathcal{E}_h}\int_e\dfrac{\partial S(P_0u-u_h)}{\partial n}\jump{p-I_{CR}p}ds\leq&Ch^{1-\delta}\|\nabla S(P_0u-u_h)\|_{L_{2-\epsilon}}\|p\|_{W^{2,\frac{2-\epsilon}{1-\epsilon}}(\Omega)},
    \end{align}
    where $\delta=\dfrac{\epsilon}{2-\epsilon}$.
    Therefore from \eqref{L2_norm_error_state_eq4} it remains to show that $\|\nabla S(P_0u-u_h)\|_{L_{2-\epsilon}}$ is uniformly bounded. 
    Using the Poisson kernel representation of $S(P_0u-u_h)$ we obtain that for any $x\in\Omega$
    \begin{equation}\label{L2_norm_error_state_eq5}
    |\nabla S(P_0u-u_h)(x)|\leq C\dfrac{|(P_0u-u_h)(\xi_0)|}{r},
    \end{equation}
    where 
    \begin{itemize}
        \item $r=\text{dist}(x,\Gamma)$.

        \item $\xi_0$ is the boundary point closest to $x$.
    \end{itemize}
    Then using tubular coordinates $x\rightarrow (\xi,r),~\xi\in\Gamma, r\in(0,r_0)$ where $r_0\leq \text{diam}(\Omega)$, for $q=2-\epsilon$ we obtain
    \begin{align}\label{L2_norm_error_state_eq6}
        \|\nabla S(P_0u-u_h)\|^q_{L_q(\Omega)}&\leq \int_{\Gamma}\int_0^{r_0}|(P_0u-u_h)(\xi)|^qr^{-q}\cdot rdrd\sigma(\xi)\notag\\
        &=\int_{\Gamma}|(P_0u-u_h)(\xi)|^qd\sigma(\xi)\int_0^{r_0)}r^{1-q}dr\notag\\
        \leq&C\|(P_0u-u_h)\|^q_{L_q(\Gamma)}
    \end{align}
    Applying \eqref{L2_norm_error_state_eq6} we say that
    \begin{align}\label{L2_norm_error_state_eq7}
        \|\nabla S(P_0u-u_h)\|_{L_q(\Omega)}&\leq C\|(P_0u-u_h)\|_{L_q(\Gamma)}\notag\\
        &\leq C\|(P_0u-u_h)\|_{L_2(\Gamma)}
    \end{align}
    Therefore \eqref{L2_norm_error_state_eq7} implies that $\|\nabla S(P_0u-u_h)\|_{L_{2-\epsilon}(\Omega)}$ is uniformly bounded with respect to the mesh parameter. 
    Therefore combining \eqref{L2_norm_error_state_eq3}, \eqref{L2_norm_error_state_eq4},  \eqref{L2_norm_error_state_eq5}, \eqref{L2_norm_error_state_eq6}, \eqref{L2_norm_error_state_eq7} we obtain that
\begin{align}\label{L2_norm_error_state_eq8}
     \sum_{_e\in\mathcal{E}_h}\int_e\dfrac{\partial S(P_0u-u_h)}{\partial n}\jump{p-I_{CR}p}ds\leq Ch^{1-\delta}\|P_0u-u_h\|\|p\|_{W^{2,\frac{2-\epsilon}{1-\epsilon}}(\Omega)},
\end{align}
  where $\delta=\dfrac{\epsilon}{2-\epsilon}$. From \eqref{L2_norm_error_state_eq1} next consider   $\sum_{T}\int_T\nabla (p-I_{CR}p)\cdot \nabla((S-S_h)(P_0u-u_h))dx$.
    \begin{align}\label{L2_norm_error_state_eq8}
        \sum_{T}\int_T\nabla (p-I_{CR}p)\cdot \nabla((S-S_h)(P_0u-u_h))dx\leq&\sum_T\|\nabla(p-I_{CR}p)\|_{L_{\frac{2-\epsilon}{1-\epsilon}}(T)}\notag\\
        &\|\nabla((S-S_h)(P_0u-u_h))\|_{L_{2-\epsilon}(T)}\notag\\
        \leq&\left[\sum_T\|\nabla(p-I_{CR}p)\|^{\frac{2-\epsilon}{1-\epsilon}}_{L_{\frac{2-\epsilon}{1-\epsilon}}(T)}\right]^{\frac{1-\epsilon}{2-\epsilon}}\notag\\
        &\left[\sum_T\|\nabla((S-S_h)(P_0u-u_h))\|^{2-\epsilon}_{L_{2-\epsilon}(T)}\right]^{\frac{1}{2-\epsilon}}
    \end{align}
    Next we estimate the final term of \eqref{L2_norm_error_state_eq8}.
    \begin{align}\label{L2_norm_error_state_eq9}
        \left[\sum_T\|\nabla((S-S_h)(P_0u-u_h))\|^{2-\epsilon}_{L_{2-\epsilon}(T)}\right]^{\frac{1}{2-\epsilon}}\leq& \left[\sum_T\|\nabla S(P_0u-u_h))\|^{2-\epsilon}_{L_{2-\epsilon}(T)}\right]^{\frac{1}{2-\epsilon}}+\notag\\&\left[\sum_T\|\nabla S_h(P_0u-u_h))\|^{2-\epsilon}_{L_{2-\epsilon}(T)}\right]^{\frac{1}{2-\epsilon}}
    \end{align}
    From \eqref{L2_norm_error_state_eq6} it is obvious that the first term in the right hand side of \eqref{L2_norm_error_state_eq9} is uniformly bounded. Consider now the second term $\left[\sum_T\|\nabla S_h(P_0u-u_h))\|^{2-\epsilon}_{L_{2-\epsilon}(T)}\right]^{\frac{1}{2-\epsilon}}$.
    We note that as $S_h(P_0u-u_h)$ is the discrrete harmonic lifting of $P_0u-u_h$ therefore
    \begin{align}\label{L2_norm_error_state_eq10}
        S_h(P_0u-u_h)=w^0_h+\tilde{w}_h,
    \end{align}
where $\tilde{w}_h\in V_h$ is the boundary lifting of $P_0u-u_h\in L_2(\Gamma)$, and $w^0_h\in V^0_h$ satisfies the following
\begin{align}\label{L2_norm_error_state_eq11}
    a_{pw}(w^0_h,v_h)=-a_{pw}(\tilde{w}_h,v_h),~\forall v_h\in V^0_h.
\end{align}
We know from the standard CR lifting property that 
\begin{align}\label{L2_norm_error_state_eq12}
    \|\tilde{w}_h\|_{pw}\leq C\|P_0\bar{u}-u_h\|_{L_2(\Gamma)}.
\end{align}
Next test \eqref{L2_norm_error_state_eq11} with the test function $w^0_h$ and applying the coercivity and boundedness of $a_{pw}$ we obtain
\begin{align}\label{L2_norm_error_state_eq13}
    \|w^0_h\|_{pw}\leq C\|P_0\bar{u}-u_h\|_{L_2(\Gamma)}.
\end{align}
Combining \eqref{L2_norm_error_state_eq10}, \eqref{L2_norm_error_state_eq12}, \eqref{L2_norm_error_state_eq13} we obtain
\begin{align}\label{L2_norm_error_state_eq14}
    \|S_h(P_0u-u_h)\|_{pw}\leq C\|P_0\bar{u}-u_h\|_{L_2(\Gamma)}.
\end{align}
On a triangle $T$ application of Holder's inequality yields,
\begin{align}\label{L2_norm_error_state_eq15}
   \|\nabla S_h(P_0u-u_h)\|^{2-\epsilon}_{L_{2-\epsilon}(T)}\leq& |T|^{(2-\epsilon)(\frac{1}{2-\epsilon}-\frac{1}{2})}\|\nabla S_h(P_0u-u_h)\|^{2-\epsilon}_{L_2(T)}\notag\\
   =&|T|^{\frac{\epsilon}{2}}\|\nabla S_h(P_0u-u_h)\|^{2-\epsilon}_{L_2(T)}
\end{align}
Summing over all the triangles imply
\begin{align}\label{L2_norm_error_state_eq16}
     \sum_T\|\nabla S_h(P_0u-u_h)\|^{2-\epsilon}_{L_{2-\epsilon}(T)}\leq& C\sum_T|T|^{\frac{\epsilon}{2}}\|\nabla S_h(P_0u-u_h)\|^{2-\epsilon}_{L_2(T)}\notag\\
     \leq&h^{\epsilon}\sum_T\|\nabla S_h(P_0u-u_h)\|^{2-\epsilon}_{L_2(T)}\notag\\
     \leq&h^{\epsilon}\left [ \sum_T\|\nabla S_h(P_0u-u_h)\|^{2}_{L_2(T)}\right ]^{\frac{2-\epsilon}{2}}\left[\dfrac{1}{h^2}\right ]^{\frac{\epsilon}{2}}\notag\\
     \leq&\left [ \sum_T\|\nabla S_h(P_0u-u_h)\|^{2}_{L_2(T)}\right ]^{\frac{2-\epsilon}{2}}
\end{align}
    Therefore from \eqref{L2_norm_error_state_eq14}, \eqref{L2_norm_error_state_eq16} we obtain
    \begin{align}\label{L2_norm_error_state_eq17}
         \left [\sum_T\|\nabla S_h(P_0u-u_h)\|^{2-\epsilon}_{L_{2-\epsilon}(T)}\right ]^{\frac{1}{2-\epsilon}}\leq C\|S_h(P_0u-u_h)\|_{pw}\leq C\|P_0\bar{u}-u_h\|_{L_2(\Gamma)}.
    \end{align}
    Therefore $ \left [\sum_T\|\nabla S_h(P_0u-u_h)\|^{2-\epsilon}_{L_{2-\epsilon}(T)}\right ]^{\frac{1}{2-\epsilon}}$ is uniformly bounded.
    Therefore combining \eqref{L2_norm_error_state_eq8}, \eqref{L2_norm_error_state_eq17} and standard CR interpolation error estimate we obtain
    \begin{align}\label{L2_norm_error_state_eq18}
         \sum_{T}\int_T\nabla (p-I_{CR}p)\cdot \nabla((S-S_h)(P_0u-u_h))dx\leq Ch^{1-\epsilon}\|P_0\bar{u}-u_h\|_{L_2(\Gamma)}\|p\|_{H^2(\Omega)},
    \end{align}
    for any $\epsilon>0$.
    Next consider the final term on the right hand side of \eqref{L2_norm_error_state_eq1}.
\begin{align}\label{L2_norm_error_state_eq19}
    \sum_{e}\int_{ e}\mean{\dfrac{\partial p}{\partial n}-\dfrac{\partial I_{CR}p}{\partial n}}\jump{(S-S_h)(P_0u-u_h)}ds=&\sum_{e\in\mathcal{E}^i_h}\int_{ e}\mean{\dfrac{\partial p}{\partial n}-\dfrac{\partial I_{CR}p}{\partial n}}\jump{(S-S_h)(P_0u-u_h)}ds+\notag\\
    &\sum_{e\in\mathcal{E}^b_h}\int_{ e}\mean{\dfrac{\partial p}{\partial n}-\dfrac{\partial I_{CR}p}{\partial n}}\jump{(S-S_h)(P_0u-u_h)}ds
\end{align}
Consider the first term on the right hand side of \eqref{L2_norm_error_state_eq19}.
\begin{align}\label{L2_norm_error_state_eq20}
    \sum_{e\in\mathcal{E}^i_h}\int_{ e}\mean{\dfrac{\partial p}{\partial n}-\dfrac{\partial I_{CR}p}{\partial n}}\jump{(S-S_h)(P_0u-u_h)}ds=&-\sum_{e\in\mathcal{E}^i_h}\int_{ e}\mean{\dfrac{\partial p}{\partial n}-\dfrac{\partial I_{CR}p}{\partial n}}\jump{S_h(P_0u-u_h)}ds\notag\\
    \leq&C\sum_{e\in\mathcal{E}^i_h}h\|\nabla S_h(P_0u-u_h)\|_{L_2(\omega_e)}\|p\|_{H^2(\omega_e)}\notag\\
    \leq&Ch\| S_h(P_0u-u_h)\|_{pw}\|p\|_{H^2(\Omega)}\notag\\\leq&Ch\|u_h\|_{L_2(\Gamma)}\|p\|_{H^2(\Omega)}.
\end{align}
Next consider 
\begin{align}\label{L2_norm_error_state_eq21}
&\sum_{e\in\mathcal{E}^b_h}\int_{ e}\mean{\dfrac{\partial p}{\partial n}-\dfrac{\partial I_{CR}p}{\partial n}}\jump{(S-S_h)(P_0u-u_h)}ds=\notag\\
&\sum_{e\in\mathcal{E}^b_h}\int_{ e}\mean{\dfrac{\partial p}{\partial n}-\dfrac{\partial I_{CR}p}{\partial n}} \left (\jump{(S-I_{CR}S)(P_0u-u_h)}+\jump{(I_{CR}S-S_h)(P_0u-u_h)}\right )ds\notag\\
\end{align}
Consider the second term on the right hand side of \eqref{L2_norm_error_state_eq21}.
\begin{align}\label{L2_norm_error_state_eq22}
&\sum_{e\in\mathcal{E}^b_h}\int_{ e}\mean{\dfrac{\partial p}{\partial n}-\dfrac{\partial I_{CR}p}{\partial n}} \left ( \jump{(I_{CR}S-S_h)(P_0u-u_h)}\right )ds\leq\notag\\
    &\sum_{e\in\mathcal{E}^b_h}\norm{\dfrac{\partial p}{\partial n}-\dfrac{\partial I_{CR}p}{\partial n}}_{L_{\frac{2-\epsilon}{1-\epsilon}}(e)} \norm{(I_{CR}S-S_h)(P_0u-u_h)}_{L_{2-\epsilon}(e)}\leq\notag\\
    &\left [\sum_{e\in\mathcal{E}^b_h}\norm{\dfrac{\partial p}{\partial n}-\dfrac{\partial I_{CR}p}{\partial n}}^{\frac{2-\epsilon}{1-\epsilon}}_{L_{\frac{2-\epsilon}{1-\epsilon}}(e)}\right]^{\frac{1-\epsilon}{2-\epsilon}}\left [\sum_{e\in\mathcal{E}^b_h}\norm{(I_{CR}S-S_h)(P_0u-u_h)}^{2-\epsilon}_{L_{2-\epsilon}(e)}\right]^{\frac{1}{2-\epsilon}}
\end{align}
Standard theory for CR interpolation theory and \eqref{L2_norm_error_state_eq22} implies
\begin{align}\label{L2_norm_error_state_eq23}
    \sum_{e\in\mathcal{E}^b_h}\int_{ e}\mean{\dfrac{\partial p}{\partial n}-\dfrac{\partial I_{CR}p}{\partial n}}\jump{(S-S_h)(P_0u-u_h)}ds&\leq Ch^{\frac{1}{2}}\|p\|_{W^{2,\frac{2-\epsilon}{1-\epsilon}}(\Omega)}\left [\sum_{e\in\mathcal{E}^b_h}\norm{(I_{CR}S-S_h)(P_0u-u_h)}^{2-\epsilon}_{L_{2-\epsilon}(e)}\right]^{\frac{1}{2-\epsilon}}\notag\\
    &\leq  Ch^{\frac{1}{2}}\|p\|_{W^{2,\frac{2-\epsilon}{1-\epsilon}}(\Omega)}\norm{(I_{CR}S-S_h)(P_0u-u_h)}_{L_{2-\epsilon}(\Gamma)}\notag\\
    &\leq  Ch^{\frac{1}{2}}\|p\|_{W^{2,\frac{2-\epsilon}{1-\epsilon}}(\Omega)}\norm{(I_{CR}S-S_h)(P_0u-u_h)}_{L_{2}(\Gamma)}
\end{align}

Note that $\int_e(I_{CR}S-S_h)(P_0u-u_h)ds=0$ for $e\in\mathcal{E}^b_h$, therefore an application of the Poincare's inequality on the boundary edge $e$ with appropriate scaling proves that
\begin{align}\label{L2_norm_error_state_eq24}
\left[\sum_{e\in\mathcal{E}^b_h}\norm{(I_{CR}S-S_h)(P_0u-u_h)}^2_{L_2(e)}\right]^{\frac{1}{2}}\leq& Ch^{\frac{1}{2}}\left [\sum_{e\in\mathcal{E}^b_h}\|\nabla (I_{CR}S-S_h)(P_0u-u_h)\|^2_{L_2(T_e)}\right]^{\frac{1}{2}}\notag\\
\leq& Ch^{\frac{1}{2}}\left[\sum_{e\in\mathcal{E}^b_h}\|\nabla I_{CR}S(P_0u-u_h)\|^2_{L_{2}(T_e)}\right]^{\frac{1}{2}}+\notag\\&Ch^{\frac{1}{2}}\left[\sum_{e\in\mathcal{E}^b_h}\|\nabla S_h(P_0u-u_h)\|^2_{L_{2}(T_e)}\right]^{\frac{1}{2}}\notag\\
\leq& Ch^{\frac{1}{2}-\epsilon}\left[\sum_{e\in\mathcal{E}^b_h}\|\nabla I_{CR}S(P_0u-u_h)\|^2_{L_{2-\epsilon}(T_e)}\right]^{\frac{1}{2-\epsilon}}+\notag\\&Ch^{\frac{1}{2}}\left[\sum_{e\in\mathcal{E}^b_h}\|\nabla S_h(P_0u-u_h)\|^2_{L_{2}(T_e)}\right]^{\frac{1}{2}}\notag\\
\leq& Ch^{\frac{1}{2}-\epsilon}\left[\sum_{e\in\mathcal{E}^b_h}\|\nabla S(P_0u-u_h)\|^2_{L_{2-\epsilon}(T_e)}\right]^{\frac{1}{2-\epsilon}}+\notag\\&Ch^{\frac{1}{2}}\left[\sum_{e\in\mathcal{E}^b_h}\|\nabla S_h(P_0u-u_h)\|^2_{L_{2}(T_e)}\right]^{\frac{1}{2}}\notag\\
\leq&Ch_e^{\frac{1}{2}-\epsilon}\|P_0u-u_h\|_{L_2(\Gamma)}
\end{align}
Therefore combining \eqref{L2_norm_error_state_eq22}, \eqref{L2_norm_error_state_eq23} and \eqref{L2_norm_error_state_eq24} we obtain that
\begin{align}\label{L2_norm_error_state_eq25}
    \sum_{e\in\mathcal{E}^b_h}\int_{ e}\mean{\dfrac{\partial p}{\partial n}-\dfrac{\partial I_{CR}p}{\partial n}} \left ( \jump{(I_{CR}S-S_h)(P_0u-u_h)}\right )ds\leq Ch^{1-\epsilon}\|p\|_{W^{2,\frac{2-\epsilon}{1-\epsilon}}(\Omega)}\|P_0u-u_h\|_{L_2(\Gamma)}.
\end{align}
Consider the first term on the right hand side of \eqref{L2_norm_error_state_eq21}.
\begin{align}\label{L2_norm_error_state_eq26}
    \sum_{e\in\mathcal{E}^b_h}\int_{ e}\mean{\dfrac{\partial p}{\partial n}-\dfrac{\partial I_{CR}p}{\partial n}} \left ( \jump{(S-I_{CR}S)(P_0u-u_h)}\right )ds\leq&Ch^{1+\epsilon}\|p\|_{W^{2,\frac{2-\epsilon}{1-\epsilon}}(\Omega)}\|\nabla S(P_0u-u_h)\|_{L_{2-\epsilon}(\Omega)}\notag\\
    \leq&Ch^{1+\epsilon}\|p\|_{W^{2,\frac{2-\epsilon}{1-\epsilon}}(\Omega)}\|P_0u-u_h\|_{L_{2}(\Gamma)}.
\end{align}
Combining \eqref{L2_norm_error_state_eq19}, \eqref{L2_norm_error_state_eq20}, \eqref{L2_norm_error_state_eq21}, \eqref{L2_norm_error_state_eq22}, \eqref{L2_norm_error_state_eq23}, \eqref{L2_norm_error_state_eq24}, \eqref{L2_norm_error_state_eq25}, \eqref{L2_norm_error_state_eq26} we obtain
\begin{align}\label{L2_norm_error_state_eq27}
     \sum_{e}\int_{ e}\mean{\dfrac{\partial p}{\partial n}-\dfrac{\partial I_{CR}p}{\partial n}}\jump{(S-S_h)(P_0u-u_h)}ds\leq Ch^{1-\epsilon}\|p\|_{W^{2,\frac{2-\epsilon}{1-\epsilon}}(\Omega)}\|P_0u-u_h\|_{L_{2}(\Gamma)}.
\end{align}
Therefore combining \eqref{L2_norm_error_state_eq8}, \eqref{L2_norm_error_state_eq18} and \eqref{L2_norm_error_state_eq27} we obtain the result.
\end{proof}
\begin{remark}\label{rem1}
    If we look closely it is easy to see that the proof of theorem \ref{L2_norm_error_state_eq} involves duality technique. Therefore the technique of the proof of this theorem actually implies that
    \begin{align}\label{L2_norm_error_state_eq29}
        \norm{Su_h-S_hu_h}\leq Ch^{1-\epsilon}\|u_h\|_{L_2(\Gamma)},
    \end{align}
    for any $\epsilon>0$.
\end{remark}
From \eqref{ctskkt:6} we obtained that the optimal control $\bar{u}\in H^{\frac{1}{2}}(\Gamma)$, therefore $S\bar{u}\in H^1(\Omega)$. Therefore while estimating $\norm{S\bar{u}-S_h\bar{u}}$ one expects $\norm{S\bar{u}-S_h\bar{u}}\leq Ch\|\bar{u}\|_{H^{\frac{1}{2}}(\Gamma)}$. This is discussed in the next theorem and next remark.
\begin{theorem}\label{L2_norm_error_state_eq30}
    If $S,~S_h$ denotes the continuous and discrete control to state operators respectively then
    the following estimate holds:
    \begin{align}
        |(S_h\bar{u}-S\bar{u},S_hP_0\bar{u}-S_h\bar{u}_h)|\leq Ch\norm{\bar{u}}_{H^{\frac{1}{2}}(\Gamma)}.
    \end{align}
\end{theorem}
\begin{proof}
    We employ the duality argument to obtain the estimate.
    Consider 
    \begin{align}\label{L2_norm_error_state_eq31}
        -\Delta\phi&=S_hP_0\bar{u}-S_h\bar{u}_h~\text{in}~\Omega\\
        \phi&=0~\text{on}~\Gamma\notag
    \end{align}
    Multiply \eqref{L2_norm_error_state_eq31} by $S_h\bar{u}-S\bar{u}$ and perform element-wise integration to obtain
    \begin{align}\label{L2_norm_error_state_eq32}
        (S_h\bar{u}-S\bar{u},S_hP_0\bar{u}-S_h\bar{u}_h)=&\sum_T(\nabla\phi,\nabla(S_h\bar{u}-S\bar{u}))-\sum_{e\in\mathcal{E}^i_h}\int_e\dfrac{\partial\phi}{\partial n}\jump{S_h\bar{u}ds-S\bar{u}}ds-\notag\\
        &\sum_{e\in\mathcal{E}^b_h}\int_e\dfrac{\partial\phi}{\partial n}\jump{S_h\bar{u}-S\bar{u}}ds\notag\\
        =&\sum_T(\nabla\phi,\nabla S_h\bar{u})-\sum_{e\in\mathcal{E}^i_h}\int_e\mean{\dfrac{\partial\phi}{\partial n}}\jump{S_h\bar{u}}ds-\sum_{e\in\mathcal{E}^b_h}\int_e\mean{\dfrac{\partial\phi}{\partial n}}\jump{S_h\bar{u}-\bar{u}}ds.
    \end{align}
    Note that on a boundary edge $e$ 
    \begin{align}\label{L2_norm_error_state_eq33}
        \int_eS_h\bar{u}ds=\int_e\bar{u}ds,
    \end{align}
    Therefore combining \eqref{L2_norm_error_state_eq32} and \eqref{L2_norm_error_state_eq33} one obtains
\begin{align}\label{L2_norm_error_state_eq34}
     (S_h\bar{u}-S\bar{u},S_hP_0\bar{u}-S_h\bar{u}_h)=&\sum_T(\nabla\phi,\nabla S_h\bar{u})-\sum_{e\in\mathcal{E}^i_h}\int_e\mean{\dfrac{\partial\phi}{\partial n}-\dfrac{\partial\phi_h}{\partial n}}\jump{S_h\bar{u}}ds-\notag\\
     &\sum_{e\in\mathcal{E}^b_h}\int_e\mean{\dfrac{\partial\phi}{\partial n}-\dfrac{\partial\phi_h}{\partial n}}\jump{S_h\bar{u}-\bar{u}}ds.
\end{align}
Note that if $e\in\mathcal{E}^b_h$ then
\begin{align}\label{L2_norm_error_state_eq35}
    \int_e(S_h\bar{u}-\bar{u})ds= \int_e(I_{CR}\bar{u}-\bar{u})ds=0.
\end{align}
Therefore \eqref{L2_norm_error_state_eq35} implies that $S_h\bar{u}-\bar{u}$ has zero integral mean property. Therefore on a boundary edge $e$ we have
\begin{align}\label{L2_norm_error_state_eq36}
    \|S_h\bar{u}-\bar{u}\|_{L_2(e)}\leq Ch^{\frac{1}{2}}\|\nabla S\bar{u}\|_{L_2(T_e)}.
\end{align}
Combining \eqref{L2_norm_error_state_eq34}, \eqref{L2_norm_error_state_eq35}, \eqref{L2_norm_error_state_eq36}, trace inequality with scaling and applying the discrete Cauchy-Schwarz inequality we obtain
\begin{align}\label{L2_norm_error_state_eq37}
    (S_h\bar{u}-S\bar{u},S_hP_0\bar{u}-S_h\bar{u}_h)\leq \sum_T(\nabla\phi,\nabla S_h\bar{u})+Ch\norm{\phi}_{H^2(\Omega)}\|\bar{u}\|_{H^{\frac{1}{2}}(\Gamma)}
\end{align}

From \eqref{L2_norm_error_state_eq34} consider
\begin{align}\label{L2_norm_error_state_eq38}
    \sum_T(\nabla\phi,\nabla S_h\bar{u})=& \sum_T(\nabla\phi,\nabla (S_h\bar{u}-S\bar{u}))\notag\\
    =&\sum_T(\nabla(\phi-\phi_h),\nabla (S_h\bar{u}- S\bar{u}))+\sum_T(\nabla\phi_h,\nabla (S_h\bar{u}- S\bar{u}))\notag\\
    =&\sum_T(\nabla(\phi-\phi_h),\nabla (S_h\bar{u}- S\bar{u}))-\sum_T(\nabla\phi_h,\nabla S\bar{u}).
\end{align}
Consider the final term from \eqref{L2_norm_error_state_eq38}.
\begin{align}\label{L2_norm_error_state_eq39}
    \sum_T(\nabla\phi_h,\nabla S\bar{u})=&-\sum_T(\Delta S\bar{u},\phi_h)+\sum_T\int_{\partial T}\dfrac{\partial S\bar{u}}{\partial n}\phi_hds\notag\\
    =&\sum_{e\in\mathcal{E}^i_h}\int_e\mean{\dfrac{\partial S\bar{u}}{\partial n}}\jump{\phi_h-\phi}ds.
\end{align}
Note that $\bar{u}\in H^{\frac{1}{2}}(\Gamma)$ therefore $S\bar{u}\in H^1(\Omega)$. As $S\bar{u}$ is the harmonic lifting of $\bar{u}$, therefore $\nabla S\bar{u}\in H(div,\Omega)$, therefore for any triangle $T$ in the triangulation $\nabla S\bar{u}|_T\in H(div,T)$. Therefore applying the trace inequality with scaling for $H~ div$ spaces and $H^1$ spaces functions in \eqref{L2_norm_error_state_eq39} we obtain
\begin{align}\label{L2_norm_error_state_eq40}
    \sum_T(\nabla\phi_h,\nabla S\bar{u})\leq \sum_{e\in\mathcal{E}^i_h}\left[h_e^{-\frac{1}{2}}\norm{\nabla S\bar{u}}_{L_2(T_e)}+h_e^{\frac{1}{2}}\norm{\Delta S\bar{u}}_{L_2(T_e)} \right]\left[ h_e^{-\frac{1}{2}}\norm{\phi_h-\phi}_{L_2(T_e)}+h_e^{\frac{1}{2}}\norm{\nabla(\phi_h-\phi)}_{L_2(T_e)}\right].
\end{align}
Hence an application of discrete Cauchy-Schwarz inequality and standard CR interpolation theory in \eqref{L2_norm_error_state_eq40} yields,
\begin{align}\label{L2_norm_error_state_eq41}
    \sum_T(\nabla\phi_h,\nabla S\bar{u})\leq Ch\|\nabla S\bar{u}\|\|\phi\|_{H^2(\Omega)}.
\end{align}
Finally combining \eqref{L2_norm_error_state_eq37}, \eqref{L2_norm_error_state_eq38} and \eqref{L2_norm_error_state_eq41} and elliptic regularity estimate for \eqref{L2_norm_error_state_eq31} yields
\begin{align}\label{L2_norm_error_state_eq42}
    (S_h\bar{u}-S\bar{u},S_hP_0\bar{u}-S_h\bar{u}_h)&\leq Ch\norm{S\bar{u}}_{H^1(\Omega)}\norm{\phi}_{H^2(\Omega)}\notag\\
    &\leq Ch\norm{\bar{u}}_{H^{\frac{1}{2}}((\Gamma))}\norm{S_hP_0\bar{u}-S_h\bar{u}_h}
\end{align}
Next it remains to show that
$\norm{S_hP_0\bar{u}-S_h\bar{u}_h}$ is uniformly bounded. Infact we shall show that $\norm{S_hP_0\bar{u}-S_h\bar{u}_h}\leq C\norm{P_0\bar{u}-\bar{u}_h}$.
Let,
\begin{align}\label{L2_norm_error_state_eq43}
    S_h(P_0\bar{u}-\bar{u}_h)=w^0_h+\tilde{w}_h,
\end{align}
where $\tilde{w}_h$ is the boundary lifting of $P_0\bar{u}-\bar{u}_h$ and $w^0_h\in V^0_h$ satisfies:
Find $w^0_h\in V^0_h$ such that
\begin{align}\label{L2_norm_error_state_eq44}
    a_{pw}(w^0_h,v_h)=-a(\tilde{w}_h,v_h),~\forall v_h\in V^0_h.
\end{align}
Problem \eqref{L2_norm_error_state_eq44} is wellposed by Lax-Milgram lemma. By testing \eqref{L2_norm_error_state_eq44} with $w^0_h$ and applying the coercivity and boundedness of $a_{pw(\cdot,\cdot)}$ we obtain
\begin{align}\label{L2_norm_error_state_eq45}
    \norm{w^0_h}_{pw}\leq C\norm{\tilde{w}_h}.
\end{align}
Hence the stability of the CR extension proves that
\begin{align}\label{L2_norm_error_state_eq46}
    \norm{w^0_h}_{pw}\leq C\norm{P_0\bar{u}-\bar{u}_h}.
\end{align}
Then applying Poincare's inequality we obtain
\begin{align}\label{L2_norm_error_state_eq47}
     \norm{w^0_h}\leq C\norm{P_0\bar{u}-\bar{u}_h}.
\end{align}
Also note that as $\tilde{w}_h\in V_h$ is the boundary lifting for $P_0\bar{u}-\bar{u}_h$ therefore for any interior edge the degrees of freedoms of $\tilde{w}_h$ are zero. Therefore $\tilde{w}_h|_T=0$ for all the triangles $T$ which does not contain a boundary edge. Therefore the calculations similar to the ones done in Theorem \ref{estimate_for_boundary-extension} helps us to obtain
\begin{align}\label{L2_norm_error_state_eq48}
    \norm{\tilde{w}_h}\leq C\norm{P_0\bar{u}-\bar{u}_h}.
\end{align}
Therefore combining \eqref{L2_norm_error_state_eq43}, \eqref{L2_norm_error_state_eq47}, \eqref{L2_norm_error_state_eq48} we obtain
\begin{align}\label{L2_norm_error_state_eq49}
    \norm{S_h(P_0\bar{u}-\bar{u}_h)}\leq C\norm{P_0\bar{u}-\bar{u}_h}.
\end{align}
Therefore combining \eqref{L2_norm_error_state_eq42} and \eqref{L2_norm_error_state_eq49} we obtain the desired result.
\begin{remark}\label{rem2}
    If we look closely it is easy to see that the proof of theorem \ref{L2_norm_error_state_eq30} involves duality technique. Therefore the technique of the proof of this theorem actually implies that
    \begin{align}\label{L2_norm_error_state_eq29}
        \norm{S\bar{u}-S_h\bar{u}}\leq Ch\|u\|_{H^{\frac{1}{2}}(\Gamma)},
    \end{align}
\end{remark}

\end{proof}


To begin with, we define the following spaces:
\begin{align*}
W_h&:=\{v_h\in H^1(\O):v_h|_{T}\in \mathbb{P}_1(T) ~\forall ~ T\in \cT_h\},\\
U_h&:=\{u_h\in L_2(\Gamma):u_h|_{e}\in \mathbb{P}_0(e) ~ \forall ~ e\in \cE_h^b\},\\
 V_{1}(\Gamma)&:=\{v_{h}\in C(\Gamma): v_{h}|_{e}\in \mathbb{P}_{1}(e)~ \forall ~ e\in \cE_h^b\}.
\end{align*}

\section{ Error Estimate for the Optimal Control}
In this section we derive the main result of this article $i.~e.$ the $L_2(\Gamma)$ norm error estimate for the optimal control.
\begin{theorem}
    Let $\bar{u}$ denotes the solution of the problem \eqref{Costfnl}-\eqref{eqn:state} and $\bar{u}_h$ of \eqref{dcp}-\eqref{dse}. Then the following estimate holds:
    \begin{align*}
        \norm{\bar{u}-\bar{u}_h}_{L_2(\Gamma)}\leq Ch^{\frac{1}{2}-\epsilon},
    \end{align*}
    for any $\epsilon>0$.
\end{theorem}
\begin{proof}
    We begin by choosing $\bar{u}_h$ as a test function in \eqref{ctskkt:5}  to obtain
    \begin{align}\label{main_result1}
       0\leq& (\alpha\bar{u}-\dfrac{\partial p}{\partial n},\bar{u}_h-\bar{u})\notag\\
       0\leq& (\alpha\bar{u}-\alpha\bar{u}_h+\alpha\bar{u}_h-\dfrac{\partial p_h}{\partial n}+\dfrac{\partial p_h}{\partial n}-\dfrac{\partial p}{\partial n},\bar{u}_h-\bar{u})\notag\\
       \alpha\norm{\bar{u}-\bar{u}_h}^2\leq& (\alpha\bar{u}_h-\dfrac{\partial p_h}{\partial n}+\dfrac{\partial p_h}{\partial n}-\dfrac{\partial p}{\partial n},\bar{u}_h-\bar{u})\notag\\
       \alpha\norm{\bar{u}-\bar{u}_h}^2\leq&(\alpha\bar{u}_h-\dfrac{\partial p_h}{\partial n},\bar{u}_h-P_0\bar{u})+(\alpha\bar{u}_h-\dfrac{\partial p_h}{\partial n},P_0\bar{u}-\bar{u})+(\dfrac{\partial p_h}{\partial n}-\dfrac{\partial p}{\partial n},\bar{u}_h-\bar{u})\notag\\
       \alpha\norm{\bar{u}-\bar{u}_h}^2\leq&(\dfrac{\partial p}{\partial n}-\dfrac{\partial p_h}{\partial n},\bar{u}-\bar{u}_h)\notag\\
       \alpha\norm{\bar{u}-\bar{u}_h}^2\leq&(\dfrac{\partial p}{\partial n},\bar{u}-\bar{u}_h)-(\dfrac{\partial p_h}{\partial n},\bar{u}-\bar{u}_h)\notag\\
       \alpha\norm{\bar{u}-\bar{u}_h}^2\leq&(\dfrac{\partial p}{\partial n},\bar{u}-P_0\bar{u})+(\dfrac{\partial p}{\partial n},P_0\bar{u}-\bar{u}_h)-(\dfrac{\partial p_h}{\partial n},\bar{u}-P_0\bar{u})-(\dfrac{\partial p_h}{\partial n},P_0\bar{u}-\bar{u}_h)\notag\\
       \alpha\norm{\bar{u}-\bar{u}_h}^2\leq&(\dfrac{\partial p}{\partial n}-\dfrac{\partial I_{CR}p}{\partial n},\bar{u}-P_0\bar{u})+(\dfrac{\partial p}{\partial n},P_0\bar{u}-\bar{u}_h)-(\dfrac{\partial p_h}{\partial n},P_0\bar{u}-\bar{u}_h)\notag\\
       \alpha\norm{\bar{u}-\bar{u}_h}^2\leq&(\dfrac{\partial p}{\partial n}-\dfrac{\partial I_{CR}p}{\partial n},\bar{u}-P_0\bar{u})+(\dfrac{\partial p}{\partial n}-\dfrac{\partial p_h}{\partial n},P_0\bar{u}-\bar{u}_h)\notag\\
       \alpha\norm{\bar{u}-\bar{u}_h}^2\leq&(\dfrac{\partial p}{\partial n}-\dfrac{\partial I_{CR}p}{\partial n},\bar{u}-P_0\bar{u})+(\dfrac{\partial p}{\partial n},P_0\bar{u}-\bar{u}_h)-(\dfrac{\partial p_h}{\partial n},P_0\bar{u}-\bar{u}_h).
    \end{align}
    Next we analyze the terms $(\dfrac{\partial p}{\partial n},P_0\bar{u}-\bar{u}_h)$ and $(\dfrac{\partial p_h}{\partial n},P_0\bar{u}-\bar{u}_h)$ from \eqref{main_result1}.
    \vspace{1 mm}
    Let $w$ denotes the harmonic lifting of $P_0\bar{u}-\bar{u}_h$. Then from the definition \ref{control:to:state} we obtain 
    \begin{align}\label{main_result2}
        S(P_0\bar{u}-\bar{u}_h)=w
    \end{align}
     Combining \eqref{As1_strong_form} and \eqref{main_result2} we obtain
     \begin{align}\label{main_result3}
         \int_{\Gamma}\dfrac{\partial p}{\partial n}P_0\bar{u}-\bar{u}_h)ds=&-(y-y_d,S(P_0\bar{u}-\bar{u}_h))\notag\\
         =&-(S^*(y-y_d),P_0\bar{u}-\bar{u}_h).
     \end{align}
     Therefore combining Lemma \ref{discrete_solution_operator_representation}, \eqref{main_result1}, \eqref{main_result3}  we obtain
     \begin{align}\label{main_result4}
         \alpha\norm{\bar{u}-\bar{u}_h}^2\leq&(\dfrac{\partial p}{\partial n}-\dfrac{\partial I_{CR}p}{\partial n},\bar{u}-P_0\bar{u})+(S^*_h(S_h\bar{u}_h-y_d),P_0\bar{u}-\bar{u}_h)-\notag\\
         &(S^*(S\bar{u}_h-y_d),P_0\bar{u}-\bar{u}_h)+(y_h-y_d,\xi_h),
     \end{align}
     where $\xi_h\in V_h$ in \eqref{main_result4} denotes the boundary lifting of $P_0\bar{u}-\bar{u}_h\in U_h$. Then \eqref{main_result4}, Theorem \ref{estimate_for_boundary-extension} and standard CR interpolation error estimates implies
     \begin{align}\label{main_result5}
         \alpha\norm{\bar{u}-\bar{u}_h}^2\leq& Ch\left[\norm{p}_{H^2(\Omega)}\norm{\bar{u}}_{H^{\frac{1}{2}}(\Gamma)}+\norm{y_h-y_d}_{L_{\infty}(\Omega)}\right]+(S^*_h(S_h\bar{u}_h-y_d),P_0\bar{u}-\bar{u}_h)-\notag\\
         &(S^*(S\bar{u}_h-y_d),P_0\bar{u}-\bar{u}_h)\notag\\
         =&Ch\left[\norm{p}_{H^2(\Omega)}\norm{\bar{u}}_{H^{\frac{1}{2}}(\Gamma)}+\norm{y_h-y_d}_{L_{\infty}(\Omega)}\right]+(S_h\bar{u}_h-y_d,S_hP_0\bar{u}-S_h\bar{u}_h)-\notag\\
         &(S\bar{u}-y_d,SP_0\bar{u}-S\bar{u}_h)
     \end{align}
     Next we estimate the last two terms from \eqref{main_result5}.
     \begin{align}\label{main_result6}
         (S_h\bar{u}_h-y_d,S_hP_0\bar{u}-S_h\bar{u}_h)-(S\bar{u}-y_d,SP_0\bar{u}-S\bar{u}_h)=&(S_h\bar{u}_h-y_d,S_hP_0\bar{u}-S_h\bar{u}_h)-\notag\\&(S_hP_0\bar{u}-y_d,S_hP_0\bar{u}-S_h\bar{u}_h)+\notag\\
         &(S_hP_0\bar{u}-y_d,S_hP_0\bar{u}-S_h\bar{u}_h)-\notag\\&(S\bar{u}-y_d,SP_0\bar{u}-S\bar{u}_h)\notag\\
         =&-\norm{S_h\bar{u}_h-S_hP_0\bar{u}}^2+\notag\\
         &(S_hP_0\bar{u}-y_d,S_hP_0\bar{u}-S_h\bar{u}_h)\notag\\&-(S\bar{u}-y_d,SP_0\bar{u}-S\bar{u}_h)\notag\\
         \leq&(S_hP_0\bar{u}-y_d,S_hP_0\bar{u}-S_h\bar{u}_h)-\notag\\
         &(S_h\bar{u}-y_d,S_hP_0\bar{u}-S_h\bar{u}_h)+\notag\\
         &(S_h\bar{u}-y_d,S_hP_0\bar{u}-S_h\bar{u}_h)-\notag\\
         &(S\bar{u}-y_d,S_hP_0\bar{u}-S_h\bar{u}_h)+\notag\\
         &(S\bar{u}-y_d,S_hP_0\bar{u}-S_h\bar{u}_h)-\notag\\
         &(S\bar{u}-y_d,SP_0\bar{u}-S\bar{u}_h)\notag\\
         =&(S_h(P_0\bar{u}-\bar{u}),S_h(P_0\bar{u}-\bar{u}_h))+\notag\\
         &(S_h\bar{u}-S\bar{u},S_h(P_0\bar{u}-\bar{u}_h))+\notag\\
         &(S\bar{u}-y_d,(S_h-S)(P_0\bar{u}-\bar{u}_h)).
    \end{align}
    Next consider the first term on the right hand side of \eqref{main_result6}. Note that \eqref{discrete_control_to_state} implies that $S^*_h:V_h\rightarrow U_h$. Therefore
    \begin{align}\label{main_result7}
        (S_h(P_0\bar{u}-\bar{u}),S_h(P_0\bar{u}-\bar{u}_h))=(P_0\bar{u}-\bar{u},S^*_h(S_h(P_0\bar{u}-\bar{u}_h)))=0.
    \end{align}
    Therefore combining Theorem \ref{L2_norm_error_state_eq}, Theorem \ref{L2_norm_error_state_eq30}, \eqref{main_result5},  \eqref{main_result6}, \eqref{main_result7} we obtain the result.
\end{proof}

Next we prove the following auxiliary estimate which is important to obtain the optimal order error estimate for the optimal control.

\end{document}